\documentclass[notitlepage,11pt,reqno]{amsart}
\usepackage{amssymb,amsmath,amscd,xspace,amsthm}
\usepackage[mathscr]{eucal}
\usepackage[breaklinks=true]{hyperref}
\usepackage{color}
\advance\textwidth by 1.2in \advance\oddsidemargin by -.6in \advance\evensidemargin by -.6in
\DeclareMathOperator{\ad}{\ensuremath{ad}\xspace}
\DeclareMathOperator{\im}{\ensuremath{im}\xspace}
\DeclareMathOperator{\Sym}{\ensuremath{Sym}\xspace}
\DeclareMathOperator{\Hom}{\ensuremath{Hom}\xspace}
\DeclareMathOperator{\End}{\ensuremath{End}\xspace}
\DeclareMathOperator{\Mod}{\ensuremath{Mod}\xspace}
\DeclareMathOperator{\Ext}{\ensuremath{Ext}\xspace}
\DeclareMathOperator{\ev}{\ensuremath{{ev}\xspace}}
\DeclareMathOperator{\wt}{\ensuremath{{wt}\xspace}}

\DeclareMathOperator{\op}{\ensuremath{{op}\xspace}}
\DeclareMathOperator{\id}{\ensuremath{{id}\xspace}}
\DeclareMathOperator{\gldim}{\ensuremath{{gldim}\xspace}}

\newcommand{\lie}{\mathfrak}
\newcommand{\g}{\lie{g}}
\newcommand{\h}{\lie{h}}
\newcommand{\curs}{\mathcal}

\newcommand{\F}{\ensuremath{\mathcal{F}}\xspace}
\newcommand{\G}{\ensuremath{\mathcal{G}}\xspace}
\newcommand{\ghat}{\ensuremath{\widehat{\mathcal{G}} }\xspace}

\newcommand{\po}{\preccurlyeq}
\newcommand{\lpsi}{\leq_{\Psi}}
\newcommand{\popsi}{\preccurlyeq_{\Psi}}

\newcommand{\V}{\ensuremath{{\bf V}}\xspace}
\newcommand{\vstar}{\ensuremath{{\bf V}^\circledast}\xspace}
\newcommand{\A}{\ensuremath{{\bf A}}\xspace}

\newcommand{\nn}{\mathbb{Z}_+}

\newcommand{\C}{\mathbb C}
\newcommand{\Z}{\mathbb{Z}}
\newcommand{\bp}{\ensuremath{\mathbb{P}}\xspace}

\newcommand{\univ}{{\bf U}}
\newcommand{\onto}{\twoheadrightarrow}

\newcommand{\comment}[1]{}

\newtheorem*{thm}{Theorem}
\newtheorem*{prop}{Proposition}
\newtheorem*{lem}{Lemma}
\newtheorem*{cor}{Corollary}

\theoremstyle{definition}
\newtheorem*{rem}{Remark}
\newtheorem*{defn}{Definition}

\newcommand{\thmref}[1]{Theorem~\ref{#1}}
\newcommand{\lemref}[1]{Lemma~\ref{#1}}
\newcommand{\propref}[1]{Proposition~\ref{#1}}

\DeclareMathOperator{\Ob}{\ensuremath{Ob}\xspace}

\newcounter{cnt}
\newenvironment{enumerit}{\begin{list}{{\hfill\rm(\roman{cnt})\hfill}}{%
\settowidth{\labelwidth}{{\rm(iv)}}\leftmargin=\labelwidth%
\advance\leftmargin by \labelsep\rightmargin=0pt\usecounter{cnt}}}{\end{list}} \makeatletter
\def\mydggeometry{\makeatletter\dg@YGRID=1\dg@XGRID=20\unitlength=0.003pt\makeatother}
\makeatother \theoremstyle{remark}

\numberwithin{equation}{section}

\begin{document}
\title{Faces of polytopes and Koszul algebras}

\author{Vyjayanthi Chari}\thanks{V.C.~was partially supported by the NSF
grant DMS-0091253.}
\address{V.~Chari: Department of Mathematics, University of
California at Riverside}
\email{\tt chari@math.ucr.edu}

\author{Apoorva Khare}
\address{A.~Khare: Departments of Mathematics and Statistics, Stanford
University}
\email{\tt khare@stanford.edu}

\author{Tim Ridenour}
\address{T.~Ridenour.: Department of Mathematics, Northwestern
University}
\email{\tt tbr4@math.northwestern.edu}

\date{\today}

\begin{abstract}
Let $\g$ be a semisimple Lie algebra and $V$ a $\g$--semisimple module. In
this article, we study the category $\G$ of $\Z$--graded
finite-dimensional representations of $\g \ltimes V$. We show that the
simple objects in this category are indexed by an interval-finite poset
and  produce a large class of truncated subcategories which are directed
and highest weight.
In the case when $V$ is a finite--dimensional $\g$--module, we construct
a family of Koszul algebras which are indexed by certain subsets of the
set of weights $\wt(V)$ of $V$. We use these Koszul algebras to construct
an infinite-dimensional graded subalgebra $\A_{\Psi}^{\g}$ of the locally
finite part of the algebra of invariants $(\End_{\C} (\V) \otimes \Sym
V)^{\g}$, where $\V$ is the direct sum of all simple finite-dimensional
$\g$-modules. We prove that $\A_{\Psi}^{\g}$ is Koszul of finite global
dimension.
\end{abstract}
\maketitle

\section*{Introduction}

In this paper, we study the category of finite--dimensional
representations of the semi--direct product Lie algebras $\g \ltimes V$,
where $\lie g$ is a complex semisimple
Lie algebra and $V$ is a $\g$-semisimple representation. There are
several well--known classical families of such Lie algebras, for
instance, the co--minuscule parabolic subalgebras of a simple Lie
algebra. However, our primary motivation comes from two sources: the
first is the truncated current algebras $\g \otimes \C[t] / t^r\C[t]$,
where $\C[t]$ is the polynomial ring in an indeterminate $t$, and their
multi--variable generalizations; and the second motivation is our
interest in the undeformed infinitesimal Hecke algebras (see
\cite{EGG},\cite{Pan},\cite{Kh,KT}).
The representation theory of the truncated current algebras has
interesting combinatorial properties and is connected with important
families of representations of quantum affine algebras. It appears likely
that the more general setup that we consider will also have such
connections \cite{BCFM}.

In this paper, we are primarily interested in understanding the
homological properties of the category of finite--dimensional
representations of the semi--direct product $\g \ltimes V$. To be more
precise, we shall regard the Lie algebra as being graded by the
non--negative integers $\Z_+$. We assume that $\lie g$ lives in grade
zero and that $V$ is finite-dimensional and lives in grade
one. The universal enveloping algebra of $\g\ltimes V$ is also
$\Z_+$--graded and in fact has elements of grade $s$ for all $s \in
\Z_+$.
We work with the category of $\Z$--graded modules for $\g\ltimes V$,
where the morphisms are just the degree zero maps. In the special case
when $V$ is the adjoint representation of a simple Lie algebra, this
category was previously studied in \cite{CG1} and \cite{CG2}. The authors
of those papers made certain choices which were not completely understood
or explained. In the current paper, we recover as a special case the
results of \cite{CG1} and \cite{CG2} and, using the results of
\cite{KhRi}, provide a more conceptual explanation for the choices.

We now explain the overall organization of the paper. The main result,
which is the construction of a family of Koszul algebras, is given in
Section 1 and can be stated independently of the representation theory of
$\g\ltimes V$. Thus, let $\lie h$ be a Cartan subalgebra of $\lie g$ and
let $\wt(V)$ be the set of weights of $V$. We consider the convex
polytope defined by $\wt(V)$.
For each subset $\Psi$ of $\wt(V)$ which lies on a face of this polytope,
we define a $\Z$--graded $\lie g$--module algebra $\A_\Psi$. We prove
that the (infinite--dimensional) subalgebra $\A_{\Psi}^{\g}$ of $\lie
g$--invariants is Koszul of global dimension at most equal to the sum of
the dimension of eigenspaces of $V$ corresponding to the elements of
$\Psi$.
The strategy to prove this result is the following. We first observe that
the set $\Psi$ defines in a natural way a partial ordering $\le_\Psi$ on
$\lie h^*$. Associated to each element $\nu$ of $\wt(V)$, we can define a
subalgebra  $\A_{\Psi}(\le_\Psi \nu)^{\g}$ of $\A_{\Psi}^{\g}$. We relate
this algebra to the endomorphism algebra of the projective generator  of
a full subcategory (with finitely many simple objects)  of $\Z$--graded
finite--dimensional representations of $\g\ltimes V$. This is done in
Sections two through four where we analyze the homological properties of
the category. To do this, we need to work in a bigger category which has
enough projectives.
The Kozsul complex associated to the symmetric algebra of $V$ provides a
projective resolution of the simple objects and we can compute arbitrary
extensions between the simple modules. This allows us, in Section 5, to
use results of \cite{BGS}  to prove that $\A_{\Psi}(\le_\Psi \nu)^{\g}$
is Koszul. The final step is to prove that this implies that
$\A_{\Psi}^{\g}$ is Koszul. 
A natural question that arises from our work is the realization of Koszul
duals of these algebras as module categories arising from Lie theory. The
authors hope to pursue this question in the future.

\subsection*{Acknowledgments}

We thank the referee for comments and suggestions which helped bring the
paper into its present form. We are also grateful to Jacob Greenstein for
useful discussions.

\section{The Main Results}

We shall denote by $\Z$ (resp.~$\Z_+$, $\C$) the set of integers
(resp.~non--negative integers, complex numbers).

\subsection{}

Throughout this paper, we fix a complex semisimple Lie algebra $\g$ and a
Cartan subalgebra $\h$ of $\g$. We let $R$ be the set of roots of $\g$
with respect to $\h$, and fix a set of simple roots $\Delta = \{ \alpha_i
| i \in I \}$ of $R$. If $\{ \omega_i | i \in I \}$ is a set of
fundamental weights, we denote by $P^+$ the $\Z_+$--span of the
fundamental weights, and by $Q^+ $ the $\Z_+$--span of $\Delta$.

\subsection{}

Given  $\mu \in P^+$, let $V(\mu)$ be the finite-dimensional simple
$\g$--module with highest weight $\mu$. Define $$ \V := \bigoplus_{\mu
\in P^+} V(\mu) \mbox{ and } \vstar := \bigoplus_{\mu \in P^+}
V(\mu)^*.$$ The natural embedding $\vstar \otimes \V \to \End \V$
(respectively $\vstar \otimes \V \to \End \vstar$) of $\g$-modules is an
anti-homomorphism (resp., a homomorphism).
For each $\mu \in P^+$, this anti-homomorphism (resp., homomorphism)
restricts to an isomorphism $V(\mu)^* \otimes V(\mu) \to (\End
V(\mu))^{\op}$ (resp., $V(\mu)^* \otimes V(\mu) \to \End V(\mu)^*$).
Under this isomorphism, the preimage of the identity element in $(\End
V(\mu))^{\op}$ is  the canonical $\g$--invariant element $1_{\mu} \in
V(\mu)^* \otimes V(\mu)$.

\subsection{}

Suppose that $A$ is a $\Z$-graded $\g$-module algebra; i.e., $A$  is an
associative $\Z$--graded algebra which admits a compatible action of
$\lie g$:
$$A = \bigoplus_{k \in \Z} A[k],\qquad \ \g.A[k] \subset A[k], k \in
\Z.$$

\noindent The space
$$\A= A \otimes \vstar \otimes \V,$$

\noindent has a natural $\Z$--grading given by
$$\A[k] = A[k] \otimes \vstar \otimes \V.$$

\noindent Moreover, it acquires the structure of a $\Z$--graded
$\g$--module algebra as follows: the multiplication is given by linearly
extending the assignment
$$(a\otimes f \otimes v)(b \otimes g \otimes w) = g(v) ab \otimes f
\otimes w,$$

\noindent while the $\g$--module structure is just given by the usual
action on tensor products of $\g$--modules.

Abusing notation, we set
$$1_{\mu} = 1_A \otimes 1_{\mu}.$$
The next Lemma is immediate.

\begin{lem}\label{L21}
For $\mu, \nu \in P^+$, we have,
$$1_{\mu} \A 1_{\nu} = A \otimes V(\mu)^* \otimes
V(\nu).$$\hfill\qedsymbol
\end{lem}

\subsection{}

From now on, we fix a finite-dimensional $\g$--module $V$ and write
$$V=\oplus_{\mu\in\lie h^*}V_\mu,\ \ V_\mu=\{v\in V: hv=\mu(h)v\ \
\forall h\in\h\}.$$

\noindent Set $\wt(V) = \{ \mu \in \h^*: V_\mu \ne 0\}$ and assume that
$\wt(V) \neq \{ 0 \}$. We shall also set
$$V^{\lie g}=\{v\in V: x v=0, \ \ \forall\ \ x\in\lie g\}.$$

Suppose that $\Psi \subset \wt(V)$ is nonempty. Define a reflexive,
transitive relation $\lpsi$ on $\h^*$ via
$$\mu \lpsi \nu \iff \nu - \mu \in \Z_+ \Psi,$$

\noindent and set
$$d_{\Psi}(\mu,\nu) := \min \{ \sum_{\beta \in \Psi} m_{\beta} : \nu -
\mu = \sum_{\beta \in \Psi} m_{\beta} \beta, \ m_{\beta} \in \Z_+\
\forall \beta \in \Psi \}. \label{dist}$$

\subsection{}

For $\mu \lpsi \nu \in P^+$, define
$$\A_{\Psi}(\nu,\mu) := 1_{\nu}\A[d_{\Psi}(\mu,\nu)]1_{\mu},$$

\noindent and given $F \subset P^+$, define
$$ \A_{\Psi}(F) := \bigoplus_{{\mu,\nu \in F, \\ \mu \lpsi \nu}}
\A_{\Psi}(\nu,\mu).$$
Note that $\A_\Psi(F)$ is a $\lie g$--module.

\begin{lem}
Suppose that $\Psi \subset \wt(V)$ is the set of weights of $V$ which lie
on some proper face of the weight polytope of $V$, and let $A$ be as
above.
\begin{enumerit}
\item $\A_{\Psi}(F)$ is a graded subalgebra of $\A_{\Psi}(G)$ for all $F
\subset G \subset P^+$.

\item If $F \subset G \subset P^+$, then $\A_{\Psi}(F)^{\g}$ is a graded
subalgebra of $\A_{\Psi}(G)^{\g}$.
\end{enumerit}
\end{lem}

This result is clear once we show that $d_\Psi$ satisfies the following
in \propref{P72}:
\begin{equation}
d_\Psi(\eta, \mu) + d_\Psi(\mu,\nu) = d_\Psi(\eta,\nu)\ \forall \ \eta
\leq_\Psi \mu \leq_\Psi \nu \in \h^*.
\end{equation}

\subsection{}

Given $\mu,\nu \in P^+$, define $[\mu,\nu]_{\Psi} := (\lpsi \nu) \cap
(\mu \lpsi)$, where
$$\lpsi \nu := \{\eta \in P^+ : \eta \lpsi \nu\}, \mbox{ and } \mu \lpsi
:= \{\eta \in P^+ : \mu \lpsi \eta\}.$$

The main theorem of this paper is the following.

\begin{thm}\label{T1}
Suppose that $\Psi$ is a subset of $\wt(V)$ which lies on some proper
face of the weight polytope of $V$, and let $A$ be the symmetric algebra
of $V$.
Given $\mu \lpsi \nu \in P^+$, the algebras $\A_{\Psi}^{\lie g}$,
$\A_{\Psi}{(\lpsi \nu)^{\lie g}}$, $\A_\Psi{(\mu \lpsi)}^{\lie g}$, and
$\A_\Psi{([\mu,\nu]_{\Psi})^{\lie g}}$ are Koszul with global dimension
at most $N_{\Psi} := \sum_{\xi \in \Psi} \dim V_{\xi}$.
Moreover, there exist $\mu \lpsi \nu \in P^+$ such that the global
dimension of all these algebras is exactly $N_{\Psi}$.
\end{thm}

\section{The categories $\curs{G}$ and $\ghat$}

In this section, we define and study the elementary properties of the
category of $\Z$--graded finite--dimensional representations of $\g
\ltimes V$. We classify the simple objects in this category and describe
their projective covers. We denote by $\univ(\lie b)$ the universal
enveloping algebra of a Lie algebra $\lie b$. We use freely the notation
established in Section 1.

\subsection{}

We begin by working in the following general situation. Thus, we assume
that $\lie{a}$ is a $\Z_+$-graded complex Lie algebra,
$$\lie{a} = \displaystyle \bigoplus_{n \in \Z_+} \lie{a}_n,$$

\noindent with the additional assumptions that $\lie a_0=\lie g$ and
$\dim \lie{a}_n<\infty$ for all $n\in\Z_+$. Set $\lie{a}_+ = \bigoplus_{n
>0} \lie{a}_n$ and note that it is a $\Z_+$--graded ideal in $\lie a$.

Set $R^+=R\cap Q^+$ and fix a Chevalley basis $\{x_{\alpha}^{\pm}, h_i :
\alpha \in R^+, \ i \in I\}$ of $\lie g$, and
$$\lie n^\pm=\bigoplus_{\alpha\in R^+}\lie g_{\pm\alpha},\ \ \lie g=\lie
n^-\oplus\lie h\oplus\lie n^+.$$

\subsection{}

Let $\F(\g)$ be the semisimple tensor category whose objects are
finite-dimensional $\g$-modules and morphisms are maps of $\g$--modules.
The simple objects of $\F(\g)$ are just the modules $V(\lambda)$,
$\lambda\in P^+$. We shall need the fact that $V(\lambda)$ is generated
by an element $v_\lambda$ with relations:
\begin{equation}\label{defvlambda}
\lie n^+ v_\lambda=0,\ \ hv_\lambda=\lambda(h)v_\lambda,\ \
(x^-_{\alpha_i})^{\lambda(h_i)+1}v_\lambda=0,
\end{equation}
\noindent for all $h\in\lie h$ and $i\in I$.
Any object $V$  of $\F(\g)$ is a weight module, i.e.,
$$V=\bigoplus_{\mu \in \lie h^*} V_\mu,\ \ V_\mu = \{ v \in V : hv =
\mu(h)v,\ \ h \in \lie h \},$$

\noindent and we set $\wt(V) = \{\mu \in \lie h^* : V_\mu \ne 0\}.$

Let $\ghat$ be the category whose objects are $\Z$-graded
$\lie{a}$-modules $V$ with finite-dimensional graded components $V[r]$;
i.e.,
$$V=\bigoplus_{r\in\Z}V[r],\qquad V[r]\in \Ob\F(\g)\ \forall r \in \Z.$$

\noindent The morphisms in $\ghat$ are $\lie{a}$-module maps $f: V \to W$
such that $f(V[r]) \subset W[r]$ for all $r\in\Z$. For $V\in\ghat$, we
have
$$V_\mu = \bigoplus_{r\in\Z} V_\mu[r],\qquad V_\mu[r]=V_\mu\cap V[r].$$

\noindent Observe that the adjoint representation of $\lie a$ is an
object of $\ghat$.

Let $\curs{G}$ be the full subcategory of $\ghat$ given by
$$V\in\Ob\curs{G}\iff V\in\Ob\ghat,\ \ \dim V<\infty.$$

\noindent Given $V\in\Ob\F(\g)$, let $\ev(V)\in\Ob\curs{G}$ be given by
$$\ev(V)[0] = V,\ \ \ev(V)[k] = 0 \ \forall k > 0,$$

\noindent with the $\lie{a}$-module structure defined by setting
$\lie{a}_+\ev(V) =0$ and leaving the $\g$--action unchanged.
Clearly, any $\lie g$--module morphism extends to a morphism of graded
$\lie a$--modules, and we have a covariant functor $\ev: \F(\g) \to
\curs{G}$.

\subsection{}

For $r \in \Z$, define a grading shift operator  $\tau_r: \ghat \to
\ghat$ via $$\tau_r(V)[k] = V[k-r].$$ For $\lambda\in P^+$ and $r\in\Z$,
set
$$V(\lambda,r) = \tau_r\ev V(\lambda),\ \ v_{\lambda,r} = \tau_r
v_\lambda.$$

\begin{prop}\label{P31}
For $(\lambda,r)\in P^+\times\Z$ we have that $V(\lambda,r)$ is a simple
object in $\ghat$. Moreover if $(\mu,s)\in P^+\times\Z$, then,
$$V(\lambda,r)\cong_{\ghat} V(\mu,s)\iff \lambda=\mu,\ \ {\rm{and}}\   \
r=s.$$
Conversely, if $M\in \ghat$ is simple then
$$M\cong V(\lambda,r),\ \ {\rm{ for\ some}}\ \ (\lambda,r)\in
P^+\times\Z.$$
\end{prop}

\begin{proof}
The first two statements are trivial. Suppose now that $M$ is a simple
object of $\ghat$ and that $r,s\in\Z$ are such that $M[r]\ne 0$ and
$M[s]\ne 0$, and assume that $r>s$. Then the subspace $\oplus_{k \geq
r}M[k] \ne 0$ and is a proper submodule of $M$, contradicting
the fact that $M$ is simple. Hence there must exist a unique $r \in \Z$
such that $M[r]\ne 0$. In particular, we have
$$M\cong_{\ghat}\tau_r \ev M[r],$$
and also that $M[r]\cong V(\lambda)$ for some $\lambda\in P^+$. This
completes the proof of the Proposition.
\end{proof}

From now on, we set $\Lambda = P^+ \times \Z$ and observe that this set
parametrizes the set of simple objects in $\ghat$ and $\curs{G}$. Given
$V\in\Ob\ghat$, we set
$$[V:V(\lambda,r)]=\dim\Hom_{\lie g}(V(\lambda), V[r]).$$

\subsection{}

We now turn our attention to constructing projective resolutions of the
simple objects of $\ghat$. The algebra $\univ(\lie a)$ has a
$\Z_+$--grading inherited from the grading on $\lie a$: namely, the grade
of a monomial $a_1\cdots a_k$, where $a_i\in \lie a_{s_i}$, is
$s_1+\cdots +s_k$. The ideal $\univ(\lie a_+)$ is a graded ideal with
finite--dimensional graded pieces. By the Poincare--Birkhoff--Witt
Theorem, we have an isomorphism of $\Z_+$-graded vector spaces
$$\univ(\lie a)\cong\univ(\lie a_+)\otimes\univ(\lie g).$$

\noindent If we regard $\lie a_+$ as a $\lie g$--module via the adjoint
action, then we have the following result, again by using the
Poincare--Birkhoff--Witt Theorem.

\begin{prop}\label{P32}
As $\g$-modules,
$$\univ(\lie{a}_+)[k] \cong \bigoplus_{(r_1,\dots,r_k)\in \Z_+^k \ : \
\sum_{j=1}^k jr_j = k} \Sym^{r_1}(\lie{a}_1) \otimes \dots \otimes
\Sym^{r_k}(\lie{a}_k).$$\hfill\qedsymbol
\end{prop}

Given $M \in \Ob\ghat$, we can regard $\univ(\lie{a}) \otimes_{\univ(\g)}
M$ as a module for $\lie a$ by left multiplication. Moreover, if we set
$$(\univ(\lie{a}) \otimes_{\univ(\g)} M)[k] = (\univ(\lie{a}_+) \otimes
M)[k] = \bigoplus_{i\in\Z_+} (\univ(\lie{a}_+)[i] \otimes M[k-i])$$

\noindent then we have the following Corollary of Proposition \ref{P32}.

\begin{cor}\label{c33}
For all $M\in\Ob\ghat$ with $M[s]=0$ for $s\ll 0$, we have
$\univ(\lie{a}) \otimes_{\univ(\g)} M\in\Ob\ghat$. \hfill\qedsymbol
\end{cor}

\subsection{}

For $(\lambda,r)\in\Lambda$, set
$$P(\lambda,r)=\univ(\lie{a}) \otimes_{\univ(\g)}
V(\lambda,r)\in\Ob\ghat,\ \ p_{\lambda,r}=1\otimes v_{\lambda,r}.$$

\begin{prop}\label{P41}
\begin{enumerit}
\item For $(\lambda,r) \in \Lambda$, we have that $P(\lambda,r)$ is
generated as a $\Z$--graded $\lie{a}$-module by $p_{\lambda,r}$ with
defining relations
\begin{equation}\label{defplambda}
\lie{n}^+p_{\lambda,r} = 0, \ \ \ hp_{\lambda,r} =
\lambda(h)p_{\lambda,r},\ \ \ \
(x_{\alpha_i}^-)^{\lambda(h_i)+1}p_{\lambda,r} = 0,
\end{equation}

\noindent for all $h\in\lie h$ and $i\in I$.
In particular, we have that if  $V \in \Ob\ghat$, then
$$\Hom_{\ghat}(P(\lambda,r),V) \cong \Hom_{\g}(V(\lambda),V[r]).$$

\item $P(\lambda,r)$ is the projective cover of its unique irreducible
quotient $V(\lambda,r)$ in $\ghat$.

\item Let $K(\lambda,r)$ be the kernel of the morphism $P(\lambda,r)\to
V(\lambda,r)$ which maps $p_{\lambda,r}\to v_{\lambda,r}$. Then
$$K(\lambda,r)=\univ(\lie a)(\lie a_+\otimes V(\lambda,r)),$$
and hence
$$\Hom_{\ghat}(K(\lambda,r), V(\mu,s))\ne 0 \implies \Hom_{\lie g}(\lie
a_{s-r}\otimes V(\lambda), V(\mu))\ne 0.$$
\end{enumerit}
\end{prop}

\begin{proof}
It is clear that the element $p_{\lambda,r}$ generates $P(\lambda,r)$ as
a $\lie a$--module. Moreover, since $v_{\lambda,r}$ satisfies relations
$\eqref{defvlambda}$, it follows that $p_{\lambda,r}$ satisfies
\eqref{defplambda}. The fact that they are the defining relations is
immediate from the Poincare--Birkhoff--Witt Theorem. It is now  easily
seen that the map
$$\varphi\to\varphi|_{1\otimes V(\lambda,r)}$$

\noindent gives an isomorphism $\Hom_{\ghat}(P(\lambda,r),V) \cong
\Hom_{\lie g}(V(\lambda), V[r])$. Since $\univ{(\lie a_+)}[0]=\C$, we see
that $\dim P(\lambda,r)_\lambda[r]=1$, and hence $P(\lambda,r)$ has a
unique maximal graded submodule with corresponding quotient
$V(\lambda,r)$.
The fact that $P(\lambda,r)$ is projective is standard. Suppose that
there exists a projective module $P\in\Ob\ghat$ and a surjective morphism
$\psi: P\to V(\lambda,r)$; and choose $p\in P_\lambda[r]$ such that
$\psi(p)=v_{\lambda,r}$. The induced morphism $\widetilde\psi: P \to
P(\lambda,r)$ must satisfy $\widetilde\psi(p)=p_{\lambda,r}$ and, hence,
is surjective, proving that $P(\lambda,r)$ is the projective cover. This
also implies that $V(\lambda,r)$ is the quotient of $P(\lambda,r)$ by
imposing the  additional relation $\lie a_+ p_{\lambda,r}=0$.
This proves that $K(\lambda,r)$ is generated as a $\lie a_+$--module
by $\lie a_+\otimes V(\lambda,r)$. Hence if $\varphi \in
\Hom_{\ghat}(K(\lambda,r), V(\mu,s))\ne 0$, then $\varphi(\lie
a_+[s-r]\otimes V(\lambda,r))\ne 0$, and the proof of the Proposition is
complete.
\end{proof}

The following is obvious.

\begin{cor}
Suppose that $M\in\ghat$ is such that $M[s]=0$ for all $s \ll 0$. Then
$M$ is a quotient of a projective object $\bp(M)\in \ghat$.
\hfill\qedsymbol
\end{cor}

\subsection{}

Motivated by the preceding Proposition, we define a partial order on
$\Lambda$ as follows.
Say that $(\mu,s)$ {\it covers} $(\lambda,r)$ if $s > r$ and $\mu -
\lambda\in \wt\lie{a}_{s-r}$. Define $\po$ to be the transitive and
reflexive closures of this relation. In particular, if $(\lambda, r) \po
(\mu,r)$, then $\lambda = \mu$. It is easily checked that $\po$ is a
partial order on $\Lambda$.

\begin{lem}\label{L26}
Let $(\lambda,r),(\mu,s)\in \Lambda$. Then,
$$\Ext^1_{\ghat}(V(\lambda,r), V(\mu,s))\ne 0\implies (\mu,s)\ \
{\rm{covers}}\ \ (\lambda,r).$$
\end{lem}

\begin{proof}
Using Proposition \ref{P41}(ii), we see that
$$\Ext^1_{\ghat}(V(\lambda,r), V(\mu,s))\cong\Hom_{\ghat}(K(\lambda,r),
V(\mu,s)).$$
The Lemma follows from \propref{P41}(iii).
\end{proof}

\subsection{}

We can now produce a projective resolution of the simple objects of
$\ghat$. The resolution is not minimal, and, in fact, it is unclear how to
produce a minimal resolution. However, as we shall see, it is adequate to
compute extensions between the simple objects.

For $(\lambda,r) \in \Lambda$ and $j \in \Z_+$, define
$$P_j(\lambda,r) := \bp(\wedge^j(\lie{a}_+)\otimes V(\lambda,r)) =
\univ(\lie{a}_+)\otimes \wedge^j(\lie{a}_+)\otimes V(\lambda,r) \in
\ghat.$$

\noindent In particular, notice that
$$P_j(\lambda,r)[k] = 0,\ \ k < r+j,\qquad \ P_0(\lambda,r) =
V(\lambda,r).$$

\noindent For $j\ge 0$, define linear maps $d_j: P_j(\lambda,r) \to
P_{j-1}(\lambda,r)$, (where we understand that $P_{-1}(\lambda,r)=
V(\lambda,r)$) by
$$d_0: P(\lambda,r) \to V(\lambda,r),\ \ d_0(u \otimes v) = u.v,$$

\noindent for $u \in \univ(\lie{a}_+)$ and $v \in V(\lambda,r)$ and
$$d_j = D \otimes \id_{V(\lambda,r)},\ \ j>0,$$

\noindent where $D$ is the Koszul differential on the Chevalley-Eilenberg
complex \cite{ChE} for $\lie{a}_+$.

\begin{prop}\label{P44}
\begin{enumerit}
\item If $j>0$ and $[P_j(\lambda,r):V(\mu,s)] \neq 0$, then $(\lambda,r)
\prec (\mu,s)$.
\item The following is a projective resolution of $V(\lambda,r)$ in
$\ghat$:
$$\dots \stackrel{d_3}{\longrightarrow} P_2(\lambda,r)
\stackrel{d_2}{\longrightarrow} P_1(\lambda,r)
\stackrel{d_1}{\longrightarrow} P(\lambda,r)
\stackrel{d_0}{\longrightarrow} V(\lambda,r) \longrightarrow 0.$$
\end{enumerit}
\end{prop}

\begin{proof}
If $j > 0$, then $P_j(\lambda,r)[r] = 0$ and so $[P_j(\lambda,r) :
V(\lambda,r)] = 0$. The rest of the proof follows by an argument similar
to the one in \propref{P41}.
Note that $P_j(\lambda,r)$ is projective in $\ghat$ for all $j\in \Z_+$
by Corollary \ref{P41}. It is straightforward to check that $d_j$ is an
$\lie{a}$-module map, and hence a morphism in $\ghat$ for all $j \in
\Z_+$. Finally, since $d_j = D \otimes \id_{V(\lambda,r)}$ for all $j$,
it follows that the sequence is exact in $\ghat$.
\end{proof}

\subsection{}

For  $s \in \Z$, let $\ghat_{\leq s}$ be the full subcategory of
$\ghat$ satisfying
$$V\in\Ob\ghat\implies \  V[k] = 0,\ \ k>s.$$

\noindent The subcategory $\curs{G}_{\le s}$ is defined similarly. Notice
that
$$V\in\Ob\ghat, V[k]=0, k\ll 0\implies V_{\le r}\in\curs{G}_{\le r},\ \
r\in\ \Z.$$

For $V \in \Ob\ghat$, define
$$V_{>s}= \bigoplus_{k>s} V[k],\qquad V_{\leq s} = V/V_{>s},$$

\noindent and note that $V_{\leq s}\in\Ob\ghat_{\leq s}$.
If $f \in \Hom_{\ghat}(V,W)$ and $s\in\Z$, there is a natural morphism
$$f_{\leq s} \in \Hom_{\ghat_{\leq s}}(V_{\leq s}, W_{\leq s}), \ \
v+V_{>s} \stackrel{f_{\leq s}}{\longmapsto} f(v) + W_{>s}.$$

\noindent The assignment $V \mapsto V_{\leq s}$ and $f \mapsto f_{\leq
s}$ defines a full, exact, and essentially surjective functor $: \ghat
\to \ghat_{\leq s}$ for each $s\in\Z_+$.
The following is immediate from Propositions \ref{P32} and \ref{P41}.

\begin{lem}\label{projles}
For $(\lambda,r)\in\Lambda$ and $s\in\Z$, we have $P(\lambda,r)_{\le
s}=0$ if $s < r$. If $s \ge r$, then $P(\lambda,r)_{\le s}$ is the
projective cover in $\curs{G}_{\le s}$ of $V(\lambda,r)$  and $P(\lambda,
r)[s]\ne 0$ if $\lie a_+\ne 0$.\hfill\qedsymbol
\end{lem}

\subsection{}

We end this section with the following result, which shows that it is
necessary to work in $\ghat$ rather than $\curs{G}$.

\begin{lem}
$\curs{G}$ has  projective objects if and only if $\lie{a}_+ = 0$.
\end{lem}

\begin{proof}
First suppose that $\lie{a}_+ = 0$ and $\lie{a} = \lie{g}$ is semisimple.
Then $\curs{G}$ is a semisimple category and all $\Ext^1$-groups vanish.
In particular, every object in $\curs{G}$ is projective.

Suppose that  $\lie{a}_+ \neq 0$ and let  $P \in \curs{G}$ be a non--zero
projective object in $\curs{ G}$. Since $P$ is finite--dimensional, we
may assume without loss of generality that $P$ is indecomposable and maps
onto $V(\lambda,r)$ for some $(\lambda,r)\in\Lambda$. For any $s\in \Z$
such that $P\in\curs{G}_{\le s}$, we see from Lemma \ref{projles} that
there exists a surjective map from $P\to P(\lambda,r)_{\le s}$. Suppose
now that $s$ is such that $P[s-1]\ne 0$ but $P[s]=0$. Then we would have
that $P(\lambda,r)_{\le s}[s]=0$, which contradicts Lemma \ref{projles}.
\end{proof}

\section{Truncated categories}

In this section, we study certain Serre subcategories of $\ghat$, and
prove that they are directed categories with finitely many simple
objects.

\subsection{}

Given $\Gamma \subset \Lambda$, let $\ghat[\Gamma]$ be the full
subcategory of $\ghat$ consisting of all $M$ such that
$$M \in \Ob\ghat,\ \ [M:V(\lambda,r)] > 0 \implies (\lambda,r)\in
\Gamma.$$

\noindent The subcategories $\curs{G}[\Gamma]$ are defined in the obvious
way. Observe that if $(\lambda,r)\in\Gamma$, then $V(\lambda,r) \in
\ghat[\Gamma]$, and we have the following trivial result.

\begin{lem}
The isomorphism classes of simple objects of $\ghat[\Gamma]$ are indexed
by $\Gamma$.\hfill\qedsymbol
\end{lem}

\subsection{}

For $V \in \ghat$, set
\begin{gather*} V_{\Gamma}^+ := \{v \in V[r]_{\lambda} : (\lambda,r) \in
\Gamma, \ \lie{n}^+v = 0\},\\
V_{\Gamma} := \univ(\lie{g})V_{\Gamma}^+ \qquad V^{\Gamma} :=
V/V_{\Lambda\setminus\Gamma}.
\end{gather*}

It is clear that $V_\Gamma$ and $V^\Gamma$ are $\Z$--graded $\lie
g$--modules, and that they are finite--dimensional if $\Gamma$ is a
finite set.
If $f\in \Hom_{\ghat}(V,W)$ then $f(V_{\Gamma}^+) \subset W_{\Gamma}^+$,
and hence the restriction $f_\Gamma$ of $f$ to $V_\Gamma$ is an element
of $\Hom_{\lie g}(V_\Gamma, W_\Gamma)$.
Moreover, since   $f(V_{\Lambda\setminus\Gamma}) \subset
W_{\Lambda\setminus\Gamma}$, we also have a natural induced  map of $\lie
g$--modules $f^{\Gamma}:V^{\Gamma} \to W^{\Gamma}$.
It is not true in general that $V_\Gamma $ and $V^\Gamma$ are in
$\ghat[\Gamma]$. However, in the case when $V_\Gamma$ and $W_\Gamma$
(resp. $V^{\Gamma}$ and $W^{\Gamma}$) are $\lie{a}$--submodules,
$f_\Gamma$ (resp. $f^{\Gamma}$) is a morphism in $\ghat[\Gamma]$.

The following is the first step in determining a sufficient condition for
this to be true. Set
$$\Lambda(V)=\{(\lambda,r)\in\Lambda : V_\lambda[r] \ne 0\}.$$

\begin{prop}\label{P52}
Suppose $V \in \ghat$ and $\Gamma \subset \Lambda$. If $V_{\Gamma}$ is
not an $\lie{a}$-submodule of $V$, then there exist $(\nu,s) \in
\Lambda(V) \setminus \Gamma$ and $(\lambda,r) \in \Gamma \cap \Lambda(V)$
such that $(\nu,s)$ covers $(\lambda,r)$.
\end{prop}

\begin{proof}
Since $V_\Gamma$ is $\Z$--graded and generated as a $\lie g$--module by
$V_\Gamma^+$, we may assume without loss of generality that there exists
$ a\in\lie a_k$ and  $v \in V_{\Gamma}^+\cap V[r]_\lambda$ with $a.v
\not\in V_{\Gamma}$ for some $k\in\Z_+$, $r\in\Z$, $\lambda\in P^+$. Let
$U$ be a  $\g$-module complement of $V_{\Gamma}$ in $V$. Then, the
projection of $av$ onto $U$ is non--zero and so there exists $\nu \in
P^+$ such that the composition of $\lie g$--module maps,
$$\lie{a}_k \otimes V(\lambda,r) \to V \onto U \onto U[r+k] \onto
V(\nu,r+k)$$

\noindent is non--zero. Call this nonzero composite map $\xi$. Now
$\lie{a}_k.V(\lambda,r) \neq 0$, so one can show that no nonzero maximal
vector $v_\lambda \in V(\lambda,r)_\lambda$ (i.e., a weight vector killed
by $\lie{n}^+$) is killed by all of $\lie{a}_k$. Since $\lie{a}_k\otimes
\C v_{\lambda}$ is a $\univ(\lie{b}^+)$-submodule of $\lie{a}_k\otimes
V(\lambda,r)$, $\xi(\lie{a}_k \otimes \C v_\lambda)$ is a nonzero
$\univ(\lie{b}^+)$-submodule of $V(\nu,r+k)$. Since $\xi(\lie{a}_k\otimes
\C v_{\lambda})$ is finite-dimensional, it must contain a maximal weight
vector in $V(\nu,r+k)$. In particular, $v_{\nu} \in \xi(\lie{a}_k\otimes
\C v_\lambda)$ where $\C v_{\nu} = V(\nu,r+k)_{\nu}$. Hence $v_\nu \in
\lie{a}_k . v_\lambda \subset V[r+k]$, so $\nu - \lambda \in
\wt(\lie{a}_k)$, and we conclude that $(\nu,r+k)$ covers $(\lambda,r)$.
\end{proof}

\subsection{}

A subset $\Gamma$ of $\Lambda$ is said to be {\it interval-closed} if
$$(\lambda,r)\po (\nu,p)\po (\mu,s),\ \ (\lambda,r), (\mu,s) \in \Gamma
\implies (\nu,p)\in\Gamma.$$

\begin{prop}\label{P53}
Suppose $\Gamma$ is a finite interval-closed subset of $\Lambda$. Let
$V\in\Ob\ghat$.
\begin{enumerit}
\item Assume that for any $(\lambda,r)\in \Lambda(V)\setminus\Gamma$ there exists $(\mu,s) \in \Gamma$ with $(\lambda,r) \prec (\mu,s)$. Then
$V_\Gamma\in\Ob\ghat[\Gamma]$. Furthermore, if $U$ is a submodule of $V$,
then
$$ U_\Gamma,(V/U)_\Gamma\in\Ob\ghat[\Gamma],\ \ (V/U)_\Gamma\cong
V_\Gamma/U_\Gamma.$$

\item  Assume that for any $(\lambda,r)\in \Lambda(V)\setminus\Gamma$
there exists $(\mu,s) \in \Gamma$ with $(\mu,s) \po (\lambda,r)$. Then
$V^\Gamma\in\Ob\ghat[\Gamma]$. Furthermore, if $U$ is a submodule of $V$,
then
$$ U^{\Gamma}, (V/U)^{\Gamma}\in\Ob\ghat[\Gamma],\ \ (V/U)^\Gamma\cong
V^\Gamma/U^\Gamma.$$
\end{enumerit}
\end{prop}

\begin{proof}
Suppose that  $V_{\Gamma}$ is not an $\lie{a}$-module. By \propref{P52}
there exists $(\lambda,r)\in\Lambda(V)\cap\Gamma$ and
$(\nu,s)\in\Lambda(V)\setminus\Gamma$ such that $(\lambda,r)\po(\nu,s)$.
By hypothesis we can choose $(\mu,k)\in\Gamma$ with
$(\lambda,r)\po(\nu,s)\po(\mu,k)$ which contradicts the fact that
$\Gamma$ is interval-closed. Suppose now that we have a short exact
sequence $$0\to U\to V\to W\to 0$$ of objects of $\ghat$.
Since $\Lambda(V) = \Lambda(U) \cup \Lambda(W)$, it is clear that $U$ and
$W$ both satisfy the hypothesis of $(i)$ and, hence,
$U_{\Gamma},W_{\Gamma} \in \ghat[\Gamma]$ and hence the inclusion of $U$
in $V$ induces a $\ghat[\Gamma]$--morphism $U_\Gamma\to V_\Gamma$ which
is obviously injective since $U^+_\Gamma\subset V_\Gamma^+$. Similarly
$W_\Gamma$ is a quotient of $V_\Gamma$ as objects of $\ghat[\Gamma]$ and
the exactness follows by noting  that $V_{\Gamma} = U_{\Gamma}\oplus
W_{\Gamma}$ as $\g$-modules.

The proof of part $(ii)$ is similar and hence omitted.
\end{proof}

\subsection{}

We now construct projective objects and projective resolutions of simple
objects in $\ghat[\Gamma]$ when $\Gamma$ is finite and interval-closed.

\begin{prop}\label{P54}
Suppose $\Gamma\subset\Lambda$ is finite and interval-closed with respect
to $\po$, and assume that $(\lambda,r),(\mu,s) \in \Gamma$.
\begin{enumerit}
\item $P(\lambda,r)^{\Gamma}$ is the projective cover in
$\curs{G}[\Gamma]$ of  $V(\lambda,r)$.

\item We have
$$[P(\lambda,r):V(\mu,s)] = [P(\lambda,r)^{\Gamma}:V(\mu,s)] = \dim
\Hom_{\curs{G}[\Gamma]}(P(\mu,s)^{\Gamma},P(\lambda,r)^{\Gamma}).$$

\item $\Hom_{\ghat}(P(\lambda,r),P(\mu,s)) \cong
\Hom_{\curs{G}[\Gamma]}(P(\lambda,r)^{\Gamma},P(\mu,s)^{\Gamma})$.

\item For all $j \in \Z_+$, $P_j(\mu,s)^{\Gamma} \in \curs{G}[\Gamma]$.
The induced sequence
$$\dots \stackrel{d_3^{\Gamma}}{\longrightarrow} P_2(\mu,s)^{\Gamma}
\stackrel{d_2^{\Gamma}}{\longrightarrow} P_1(\mu,s)^{\Gamma}
\stackrel{d_1^{\Gamma}}{\longrightarrow} P(\mu,s)^{\Gamma}
\stackrel{d_0^{\Gamma}}{\longrightarrow} V(\mu,s) \longrightarrow 0$$
\noindent is a finite projective resolution of $V(\mu,s) \in
\curs{G}[\Gamma]$.
\end{enumerit}
\end{prop}

\begin{proof}
By \propref{P41}(iii), we see that
$$(\nu,k) \succ (\lambda,r),\ \ (\nu,k) \in\Lambda(P(\lambda,r))
\setminus \{ (\lambda,r) \}.$$

\noindent By \propref{P53}(ii) we see that $P(\lambda,r)^\Gamma \in
\G[\Gamma]$ and maps onto $V(\lambda,r)$.
Let $K=P(\lambda,r)_{\Lambda\setminus\Gamma}$; thus, in $\ghat$, we have
a short exact sequence
$$0\to K \to P(\lambda,r) \to P(\lambda,r)^{\Gamma}\to 0.$$
Applying $\Hom_{\ghat}(-,V(\mu,s))$ yields the long exact sequence
\[\cdots\to \Hom_{\ghat}(K,V(\mu,s)) \to
\Ext_{\ghat}^1(P(\lambda,r)^{\Gamma},V(\mu,s)) \to 0. \]

\noindent If $(\mu,s) \in \Gamma$, we have,
$$\Hom_{\ghat}(K,V(\mu,s)) \cong \Hom_{\g}(K[s],V(\mu)) = 0,$$
and hence we have
$$ \Ext^1_{\ghat}(P(\lambda,r)^{\Gamma},V(\mu,s)) = 0.$$

\noindent In particular, this proves that
$$\Ext^1_{\curs{G}[\Gamma]}(P(\lambda,r)^{\Gamma},V(\mu,s)) = 0,$$

\noindent and hence $P(\lambda,r)^{\Gamma}$ is a projective object of
$\curs{G}[\Gamma]$.
The proof that $P(\lambda,r)^{\Gamma}$ is the projective cover of
$V(\lambda,r)$ is similar to the proof given in Proposition \ref{P41}.

For $(ii)$, we again consider the short exact sequence
$$0 \to K \to P(\lambda,r) \to P(\lambda,r)^{\Gamma} \to 0.$$
Since $\F(\g)$ is a semisimple category, we have,
$$\dim\Hom_{\g}(V(\mu),P(\lambda,r)[s]) =
\dim\Hom_{\g}(V(\mu),P(\lambda,r)^{\Gamma}[s]) +
\dim\Hom_{\g}(V(\mu),K[s]).$$

\noindent By the definition of $K$, we have,
$$\Hom_{\g}(V(\mu),K[s]) = 0,\qquad (\mu,s)\in\Gamma,$$
and hence we get
$$[P(\lambda,r):V(\mu,s)] = [P(\lambda,r)^{\Gamma}:V(\mu,s)].$$
The second equality follows by imitating (in $\curs{G}[\Gamma]$) the
proof of the first part of \propref{P41}.

For $(iii)$, choose a nonzero $f \in \Hom_{\ghat}(P(\lambda,r),
P(\mu,s))$.
Then, $f(1 \otimes V(\lambda,r)) \not\in P(\mu,s)_{\Lambda \setminus
\Gamma}$, so $f^\Gamma \neq 0$. Thus, we have an injective map
$$\Hom_{\ghat}(P(\lambda,r), P(\mu,s)) \to
\Hom_{\G[\Gamma]}(P(\lambda,r)^\Gamma, P(\mu,s)^\Gamma).$$

\noindent Since both of these spaces have the same dimension (by part
$(ii)$ and \propref{P41}(i)), the map is an isomorphism.

For $(iv)$, first note that $P_j(\lambda,r)^{\Gamma} \in
\Ob\curs{G}[\Gamma]$ by \propref{P44}(i) and \propref{P53}(ii).
Furthermore, a similar procedure as in part $(i)$ shows that
$P_j(\lambda,r)^{\Gamma}$ is projective in $\curs{G}[\Gamma]$. The fact
that the resolution terminates after finitely many steps follows from the
fact that $\Gamma$ is finite, along with the fact that $P_j(\lambda,r)[k]
= 0$ for all $k < r+j$.
\end{proof}

\subsection{}

We recall the following definition from \cite{CPS},\cite{PSW}, where we
define a length category in the sense of \cite{Gab}.

\begin{defn}
Suppose $\curs{C}$ is an abelian $\C$--linear length category.
We say that $\curs{C}$ is {\it directed} if:
\begin{enumerit}
\item The simple objects in $\curs{C}$ are parametrized by a poset
$(\Pi,\leq)$ such that the set $\{\xi \in \Pi : \xi < \tau\}$ is finite
for all $\tau \in \Pi$.

\item For all simple objects $S(\xi),S(\tau) \in \curs{C}$,
$\Ext^1_{\curs{C}}(S(\xi),S(\tau)) \neq 0 \implies \xi < \tau$.
\end{enumerit}
\end{defn}

\noindent In the case when $(\Pi,\leq)$ is finite, a directed category is
highest weight in the sense of \cite{CPS}.

We end this section by noting that we have established that for any
subset $\Gamma$ of $\Lambda$, the category $\curs{G}[\Gamma]$ is
directed, and if $\Gamma$ is finite and interval-closed, then
$\curs{G}[\Gamma]$ is a directed highest weight category.

\section{Undeformed Infinitesimal Hecke Algebras}

For the rest of the paper, we restrict our attention to  the case when
$\lie a_k=0$ for all $k>1$ and $\lie{a}_1 = V$, where $V \in \F(\g)$ is
such that $\wt(V) \neq \{ 0 \}$. In this case, the algebra $\lie a=\lie
g\ltimes V$ and we identify $V$ with the abelian ideal $0 \ltimes V$ of
$\lie{a}$.
In particular, this means that $\univ(\lie a_+) = \Sym(V)$, and it is
immediate that
$$P_j(\lambda,r) = \univ(\lie g\ltimes V) \otimes_{\univ(\g)} \wedge^j V
\otimes V(\lambda,r)$$

\noindent is generated as an $\lie{a}$-module by the component of degree
$r+j$. This motivates our search for Koszulity in this picture.

\subsection{}

We begin by computing extensions between the simple objects.

\begin{prop}\label{P61}
For all $j \in \Z_+$ and $(\mu,r),(\nu,s) \in \Lambda$,
$$\Ext_{\ghat}^j(V(\mu,r),V(\nu,s)) \cong \left\{
\begin{array}{ll}
\Hom_{\g}(\wedge^j V \otimes V(\mu),V(\nu)), & \mbox{if } j = s-r;\\
0, & \mbox{otherwise.}
\end{array}\right.$$
\end{prop}

\begin{proof}
Truncating the projective resolution from \propref{P44}(ii) at\newline
$\dots \stackrel{d_j}{\longrightarrow} P_{j-1}(\mu,r)
\stackrel{d_{j-1}}{\longrightarrow} \im d_{j-1} \longrightarrow 0$ yields
$$\Ext^j_{\ghat}(V(\mu,r),V(\nu,s)) \cong \Ext^1_{\ghat}(\im
d_{j-1},V(\nu,s)).$$

\noindent Applying $\Hom_{\ghat}(-, V(\nu,s))$ to the short exact
sequence
\[ 0 \to \im d_j \to P_{j-1}(\mu,r) \to \im d_{j-1} \to 0 \]

\noindent yields the exact sequence
\begin{eqnarray*}
&& 0 \to \Hom_{\ghat}(\im d_{j-1}, V(\nu,s)) \to
\Hom_{\ghat}(P_{j-1}(\mu,r), V(\nu,s)) \to\\
&& \Hom_{\ghat}(\im d_j, V(\nu,s)) \to \Ext^1_{\ghat}(\im d_{j-1},
V(\nu,s)) \to 0.
\end{eqnarray*}

\noindent The result follows if we prove that
$$\Hom_{\ghat}(\im d_j,V(\nu,s)) \ne 0\implies j = s-r.$$

\noindent Suppose $f\in \Hom_{\ghat}(\im d_j,V(\nu,s))$ is nonzero and
choose $v\in \im d_j[s]$ with $f(v)\ne 0$. It is easily seen that we may
write
$$v = \sum_p (u_p\otimes 1)d_j(1\otimes w_p),\ \ u_p \in \Sym V,\ \ w_p
\in \wedge^j V \otimes V(\mu,r),$$

\noindent and hence we have
$$f(v) = \sum_p (u_p\otimes 1)f(d_j(1\otimes w_p)).$$

\noindent Since $d_j(1\otimes w_p) \in \im d_j[j+r]$ for all $p$, we see
that $f(v)\in V(\nu,s)[r+j]$ and hence $s=r+j$.

If $j = s - r$, then since $\wedge^{j-1} V \otimes V(\mu,r)$ is
concentrated in degree $s-1$ and  $P_{j-1}(\mu,r)$ is the projective
cover of $\wedge^{j-1} V \otimes V(\mu,r)$, we have
$$\Hom_{\ghat}(P_{s-r-1}(\mu,r), V(\nu,s)) = 0,$$

\noindent and:
\begin{eqnarray*}
&& \Ext_{\ghat}^{s-r}(V(\mu,r),V(\nu,s))  \cong \Ext^1_{\ghat}(\im
d_{s-r-1},V(\nu,s)) \\
& \cong & \Hom_{\ghat}(\im d_{s-r},V(\nu,s)) \cong  \Hom_{\g}((\im
d_{s-r})[s], V(\nu)) \\
& \cong & \Hom_{\g}(\wedge^{s-r} V\otimes V(\mu), V(\nu)).
\end{eqnarray*}

\noindent This completes the proof of the Proposition.
\end{proof}

\subsection{}

If $\Gamma \subset \Lambda$ is finite and interval-closed, then we can
make the following observation regarding $\Ext$-groups in the truncated
subcategory $\curs{G}[\Gamma]$.

\begin{prop}\label{P62}
Let $\Gamma$ be finite and interval-closed. For all $(\mu,r),(\nu,s) \in
\Gamma$, we have
$$\Ext^j_{\curs{G}[\Gamma]}(V(\mu,r),V(\nu,s)) \cong
\Ext^j_{\ghat}(V(\mu,r),V(\nu,s)) \ \forall j\in\Z_+.$$
\end{prop}

\begin{proof}
By \propref{P44} and \propref{P54}, we have a projective resolution
$P_{\bullet}(\mu,s)$ of each simple object $V(\mu,s)$ in
$\curs{G}[\Gamma]$ and $\ghat$. Then one shows as in \cite[Proposition
3.3]{CG2}, that for all $(\lambda,r) \in \Gamma$,
$$\Hom_{\ghat}(P_{\bullet}(\mu,s),V(\lambda,r)) \to
\Hom_{\curs{G}[\Gamma]}(P_{\bullet}(\mu,s)^{\Gamma},V(\lambda,r))$$
is an isomorphism.
\end{proof}

\subsection{}

Define $\displaystyle P(\Gamma) := \bigoplus_{(\lambda,r)\in\Gamma}
P(\lambda,r),$ and set
$$\lie{B}(\Gamma) := \End_{\ghat} P(\Gamma), \mbox{ and }
\lie{B}^{\Gamma}(\Gamma) :=
\End_{\curs{G}[\Gamma]}(P(\Gamma)^{\Gamma}).$$

\noindent Notice that $\lie{B}(\Gamma)$ is graded via
$$\lie{B}(\Gamma)[k] = \bigoplus_{(\lambda,r),(\mu,r-k)\in\Gamma}
\Hom_{\ghat}(P(\lambda,r),P(\mu,r-k)).$$

\noindent In particular, $\lie{B}(\Gamma)[0] =
\bigoplus_{(\lambda,r)\in\Gamma}\End_{\ghat}(P(\lambda,r))$.

\begin{prop}
If $\Gamma\subset \Lambda$ is finite and interval-closed, then the
category $\curs{G}[\Gamma]$ is equivalent to the category of right
modules over $\lie{B}(\Gamma)$.
\end{prop}

\begin{proof}
By \propref{P54}(iii), $\lie{B}(\Gamma) \cong \lie{B}^\Gamma(\Gamma)$ if
$\Gamma$ is finite and interval-closed; thus, it is standard
(\cite[Theorem II.1.3]{B}) that
$$\Hom_{\curs{G}[\Gamma]}(P(\Gamma)^{\Gamma},-): \curs{G}[\Gamma] \to
\Mod - \lie{B}(\Gamma)$$
is an equivalence of categories.
\end{proof}

\section{Faces of Polytopes and Koszul Algebras}

This section is devoted to proving the main theorem.

\subsection{}

We begin with a key technical observation about the set of weights which
lie on a face of the weight polytope of $V$. Namely, we wish to consider
the subsets $\Psi \subset \wt(V)$ that satisfy the following property:
\begin{align}\label{rigid}
& \mbox{If } \sum_{\alpha \in \Psi} m_\alpha \alpha = \sum_{\beta \in
\wt(V)} r_\beta \beta, \mbox{ for } m_\alpha, r_\beta \in \nn,\\
& \mbox{then } \sum_\alpha m_\alpha \leq \sum_\beta r_\beta, \mbox{ with
equality if and only if } \beta \in \Psi \mbox{ whenever } r_\beta >
0.\nonumber
\end{align}

The main result of \cite{KhRi} states that $\Psi$ satisfies \eqref{rigid}
if and only if the set $\Psi$ lies on a proper face of the weight
polytope of $V$.

\subsection{}

For our next result, recall $d_{\Psi}$  and $\leq_\Psi$ defined in
Section 1.4.

\begin{prop}\label{P72}
Suppose $\Psi \subset \wt(V)$ satisfies \eqref{rigid}.
\begin{enumerit}
\item $\lpsi$ is a partial order on $\h^*$. Moreover,
$$d_{\Psi}(\eta,\mu) + d_{\Psi}(\mu,\nu) = d_{\Psi}(\eta,\nu) \ \forall \
\eta \lpsi \mu \lpsi \nu \in \h^*.$$

\item $\Psi$ induces a refinement $\popsi$ of the partial order $\po$ on
$\Lambda = P^+ \times \Z$ via: $(\mu,r) \popsi (\lambda,s)$ if and only
if $\mu \lpsi \lambda$ and $d_{\Psi}(\mu,\lambda) = s - r$. If the
interval $[(\nu,r),(\mu,s)]_{\popsi}$ is nonempty, then
$[(\nu,r),(\mu,s)]_{\popsi} = [(\nu,r),(\mu,s)]_{\po}$.
\end{enumerit}
\end{prop}

\begin{proof}
By definition, $\lpsi$ is reflexive and transitive. To see that $\lpsi$
is anti-symmetric, let
\[\nu - \mu = \sum_{\beta \in \Psi} r_{\beta}\beta, \ \ \mu - \nu =
\sum_{\beta \in \Psi} m_{\beta}\beta, \ \ r_{\beta},m_{\beta} \in \Z_+ \
\forall \ \beta \in \Psi.\]

\noindent Then,
\[0 = \sum_{\beta \in \Psi} (r_{\beta} + m_{\beta})\beta,\]

\noindent which gives $r_{\beta} + m_{\beta} = 0$ for all $\beta \in
\Psi$ using condition \eqref{rigid}. In particular, $r_{\beta} = 0$ for
all $\beta \in \Psi$, so $\mu = \nu$. This shows that $\lpsi$ is a
partial order on $\lie{h}^*$.

Suppose that $\mu \lpsi \nu$, and let
\[\nu - \mu = \sum_{\beta\in \Psi} r_{\beta}\beta = \sum_{\beta \in \Psi}
m_{\beta}\beta, \ \ r_{\beta},m_{\beta}\in\Z_+ \ \forall \ \beta \in \Psi.\]

\noindent Applying condition \eqref{rigid} to each sum gives
\[\sum_{\beta \in \Psi} r_{\beta} \leq \sum_{\beta \in \Psi} m_{\beta}
\leq \sum_{\beta \in \Psi} r_{\beta},\]

\noindent which shows that $d_{\Psi}(\mu,\nu)$ is taken over a singleton
set. This uniqueness and the fact that $\lambda - \mu = (\lambda - \nu) +
(\nu - \mu)$ show that $d_{\Psi}(\mu,\nu) + d_{\Psi}(\nu,\lambda) =
d_{\Psi}(\mu,\lambda)$.

The fact that $\popsi$ is a partial order follows immediately from part
$(i)$.

Notice that \[(\mu,r)\preccurlyeq (\lambda,s) \iff
\lambda - \mu = \sum_{\nu \in \wt(V)} m_{\nu}\nu, \ m_{\nu} \in
\Z_+ \forall \nu \in \wt(V), \ \sum_{\nu \in \wt(V)} m_{\nu} = s-r.\]
It follows immediately that $\popsi$ is a refinement of $\po$ and that
the interval $[(\nu,r),(\mu,s)]_{\popsi}$ is a subset of
$[(\nu,r),(\mu,s)]_{\po}$ for all $(\nu,r) \popsi (\mu,s)$.

Now, suppose that \[(\nu,r) \po (\eta,t) \po (\mu,s), \ \ (\nu,r) \popsi
(\mu,s) \in \Gamma\] Then, we can write
\[\mu - \eta = \sum_{i=1}^{t-s}\beta_i, \ \ \eta - \nu = \sum_{j=1}^{s-r}
\gamma_j, \ \beta_i, \gamma_j \in \wt(V) \ \forall \ i,j.\]

Since
\[\mu - \nu = (\mu -\eta) + (\eta - \nu), \mbox{ and } (s -t) + (t - r) =
s-r = d_{\Psi}(\mu,\lambda),\]

\noindent it follows that $\beta_i,\gamma_j \in \Psi$ for all $i,j$ by
condition \eqref{rigid}. This gives \[(\nu,r) \popsi (\eta,t) \popsi
(\mu,s),\] which proves $(ii)$.
\end{proof}

\begin{rem}
In \cite{CG2}, the authors work with $V = \g_{\ad}$ and $\Psi \subset
R^+$. However, they use the partial order $\popsi'$ on
$\Lambda$ given by $(\lambda,r) \popsi' (\mu,s)$ if and only if $\mu
\lpsi \lambda$ and $d_{\Psi}(\mu,\lambda) = s-r$. We use $\popsi$
instead, because $\popsi'$ is not a refinement of the standard partial
order $\po$ on $\Lambda$.
\end{rem}

\subsection{}

We need the following well-known result.

\begin{lem} \label{L73}
Suppose $\g$ is a complex semisimple Lie algebra, $V \in \F(\g)$, and
$\lambda,\mu \in P^+$. Define $V^+ := \{ v \in V : \lie{n}^+v = 0 \}$.
\begin{enumerit}
\item $\dim \Hom_{\g}(V(\lambda),V) = \dim(V^+\cap V_{\lambda})$.

\item As vector spaces,
$$\Hom_{\g}(V\otimes V(\mu),V(\lambda)) \cong \{v \in
V_{\lambda-\mu}:(x_{\alpha_i}^+)^{\mu(h_i)+1}v =
(x_{\alpha_i}^-)^{\lambda(h_i)+1} v = 0\}.$$
\end{enumerit}\hfill\qedsymbol
\end{lem}

\subsection{}

We now discuss some results on specific sets of $\g$-module homomorphisms
which will be useful later. Recall that $\lambda \in P^+$ is said to be
{\it regular} if $\lambda(h_i) > 0$ for all $i \in I$.

\begin{lem} \label{L74}
Suppose $\Psi \subset \wt(V)$ satisfies condition \eqref{rigid}.
Define $\lambda_{\Psi} := \sum_{\mu \in \Psi} (\dim V_{\mu})\mu \in P$
and $N_{\Psi} := \sum_{\mu \in \Psi} (\dim V_{\mu})$.

\begin{enumerit}
\item If $\nu, \nu + \lambda_{\Psi} \in P^+$, then
$\dim \Hom_{\g}(\wedge^{N_{\Psi}}V\otimes V(\nu),V(\nu + \lambda_{\Psi}))
\leq 1$.

\item Given $\eta \in P^+$, there exists $\nu \in P^+$ such that $\nu,
\nu + \lambda_{\Psi} \in P^+$ are both regular, $\eta \leq \nu$, and
$$\dim \Hom_{\g}(\wedge^{N_{\Psi}}V\otimes V(\nu),V(\nu +
\lambda_{\Psi})) = 1.$$
\end{enumerit}
\end{lem}

\begin{proof}
Suppose $v_{\mu_1} \wedge \dots \wedge v_{\mu_{N_\Psi}} \in
(\wedge^{N_\Psi} V)_{\lambda_\Psi}$, where each $v_{\mu_i} \in
V_{\mu_i}$. Then
\[\mu_1 + \dots + \mu_{N_\Psi} = \lambda_\Psi = \sum_{\mu \in \Psi} (\dim
V_\mu) \mu, \ \ N_\Psi = \sum_{\mu \in \Psi} \dim V_\mu,\]
so $\mu_i \in \Psi\ \forall i$ by condition \eqref{rigid}. In particular,
$\dim (\wedge^{N_\Psi} V)_{\lambda_\Psi}  = 1$. Hence $(i)$ follows by
\lemref{L73}.

Now suppose that $(\wedge^{N_{\Psi}} V)_{\lambda_{\Psi}} = \C {\bf v}$.
Let
\[\lambda_{\Psi} = \sum_{i\in I} d_i\omega_i,  \ \ \eta = \sum_{i\in I}
c_i\omega_i.\]

Let $2\rho = \sum_{\alpha \in R^+} \alpha = 2\sum_{i\in I} \omega_i$.
Choose $k \in \Z_+$ sufficiently large such that $$c_i + 2k, c_i + d_i +
2k \in \mathbb{N}$$ and
$$(x_{\alpha_i}^+)^{c_i+2k+1}{\bf v} = (x_{\alpha_i}^-)^{c_i+d_i + 2k +
1}{\bf v} = 0 \ \forall \ i\in I.$$

\noindent Let $\nu = \eta + 2k\rho$. Then, $\eta \leq \nu$, and it
follows from \lemref{L73} that
$$\Hom_{\g}(\wedge^{N_{\Psi}}V\otimes V(\nu), V(\nu + \lambda_{\Psi}))
\cong \C{\bf v},$$
which proves $(ii)$.
\end{proof}

\subsection{} \begin{lem}\label{L75}
Fix $(\mu,r) \in \Lambda$. Suppose $\Gamma \subset \Lambda$ is finite and
interval-closed with respect to $\popsi$. Also, assume that $(\mu,r)
\popsi (\nu,s) \ \forall \ (\nu,s) \in \Gamma$.
\begin{enumerit}
\item If $\Hom_{\curs{G}[\Gamma]}(P(\nu,s)^{\Gamma},P(\nu',s')^{\Gamma})
\neq 0$, then $(\nu',s') \popsi (\nu,s)$.

\item If $\Ext^j_{\curs{G}[\Gamma]}(V(\nu',s'),V(\nu,s))\neq 0$, then
$(\nu',s')\popsi (\nu,s),\ j = d_{\Psi}(\nu',\nu)$.

\item $\gldim \curs{G}[\Gamma] \leq N_{\Psi}$, and equality holds for
some $(\mu,r) \in \Lambda$ and some $\Gamma$.
\end{enumerit}
\end{lem}

\begin{proof}\hfill
\begin{enumerit}
\item Suppose that
$\Hom_{\curs{G}[\Gamma]}(P(\nu,s)^{\Gamma},P(\nu',s')^{\Gamma}) \neq 0$.
Then, by \propref{P54} and \propref{P41},
\[\Hom_{\g}(V(\nu),\Sym^{s-s'} V \otimes V(\nu')) \neq 0.\]

\noindent Using \lemref{L73} and Steinberg's formula \cite[\S 24]{H},
$$\nu-\nu' = \sum_{i=1}^{s-s'} \xi_i, \ \ \xi_i \in \wt(V).$$

\noindent On the other hand, since $(\mu,r) \popsi (\nu,s),(\nu',s') \in
\Gamma$,
$$\nu - \mu = \sum_{j=1}^{s-r} \eta_j, \ \ \nu' - \mu = \sum_{k=1}^{s'-r}
\eta'_k, \ \ \eta_j,\eta'_k \in \Psi \ \forall j,k.$$ Combining these
gives
$$\nu-\mu = \sum_{j=1}^{s-r} \eta_j = \sum_{i=1}^{s-s'} \xi_i +
\sum_{k=1}^{s'-r} \eta'_k.$$

\noindent Finally, since $\Psi$ satisfies \eqref{rigid} and $s - r =
(s-s')+(s'-r)$, we get $\xi_i \in \Psi \ \forall i$, whence $(\nu',s')
\popsi (\nu,s)$.

\item Suppose $\Ext^j_{\curs{G}[\Gamma]}(V(\nu',s'),V(\nu,s))\neq 0$. By
Propositions \ref{P61} and \ref{P62}, $j = s-s'$ and \[\Hom_{\g}(\wedge^j
V \otimes V(\nu),V(\nu')) \neq 0.\] Using \lemref{L73}, $\nu' - \nu =
\xi_1 + \dots + \xi_j$ for some $\xi_i \in \wt(V)$. Since $\Psi$
satisfies \eqref{rigid}, an argument similar to part $(i)$ shows that
$\xi_i \in \Psi \ \forall \ i$ and $\nu \lpsi \nu'$. Finally, by
\propref{P72}, $j = d_{\Psi}(\nu, \nu')$ and, therefore, $(\nu,s) \popsi
(\nu',s')$.

\item Since $\curs{G}[\Gamma]$ is a length category, it suffices to work
with extensions between simple objects. By
Propositions \ref{P61} and \ref{P62} again, we have
\[ \Ext^j_{\curs{G}[\Gamma]}(V(\nu,s),V(\nu',s')) \neq 0 \implies
\Hom_{\g}(\wedge^j V \otimes V(\nu),V(\nu')) \neq 0, \]

\noindent so $\gldim \curs{G}[\Gamma] \leq N_{\Psi}$.

Using \lemref{L74}, \[\Hom_{\curs{G}[\Gamma]}(\wedge^{N_{\Psi}}V\otimes
V(\mu),V(\mu + \lambda_{\Psi})) \neq 0\] for some $\mu, \mu +
\lambda_\Psi \in P^+$. Let $r \in \Z$ and define
\[\Gamma := [(\mu,r),(\mu+\lambda_{\Psi},r + N_{\Psi})]_{\popsi}.\]
Then, $\gldim \curs{G}[\Gamma] = N_{\Psi}$.\hfill{\qedhere}
\end{enumerit}
\end{proof}

\subsection{}

\begin{thm}\label{T3}
Assume that $\Gamma \subset \Lambda$ is finite and interval-closed under
$\popsi$. Then, the algebra $\lie{B}(\Gamma)^{\op}$ is Koszul.
\end{thm}

\begin{proof}
We use the numerical condition from \cite[Theorem 2.11.1]{BGS} to show
Koszulity. Let $B = \lie{B}(\Gamma)^{\op}$. We note that the $|\Gamma|
\times |\Gamma|$-Hilbert matrices $H(B,t)$ of $B$ and $H(E(B),t)$ of its
Yoneda algebra $E(B)$ are lower triangular in this case.

Note from the definition of the grading that $B[0]$ is semisimple,
commutative, and spanned by pairwise orthogonal idempotents
$\{1_{(\nu,s)}:(\nu,s) \in \Gamma\}$. For each $(\nu',s') \popsi (\nu,s)
\in \Gamma$, we compute:
$$\begin{array}{l}
(H(E(B),-t)H(B,t))_{(\nu,s),(\nu',s')}\\
= \displaystyle \sum_{(\xi,l) \in \Gamma} H(E(B),-t)_{(\nu,s),(\xi,l)}
H(B,t)_{(\xi,l),(\nu',s')} \\
= \displaystyle \sum_{\nu' \lpsi \xi \lpsi \nu} (-t)^{d_{\Psi}(\xi,\nu)}
\dim \Ext_{\curs{G}[\Gamma]}^{d_{\Psi}(\xi,\nu)}(V(\xi,l),V(\nu,s)) \cdot
t^{d_{\Psi}(\nu',\xi)} [P(\nu',s')^{\Gamma}:V(\xi,l)] \\
= \displaystyle \sum_{j\geq 0} \sum_{\nu' \lpsi \xi\lpsi\nu} (-1)^j
t^{d_{\Psi}(\nu',\nu)} [P(\nu',s')^{\Gamma}:V(\xi,l)] \dim
\Ext^j_{\curs{G}[\Gamma]}(V(\xi,l),V(\nu,s)), \\
\end{array}$$

\noindent where the first equality is by definition, the second uses the
definitions of the Hilbert matrices, and the third uses Proposition
\ref{P72} and Lemma \ref{L75}.
Now use the long exact sequence of Ext groups and a Jordan-Holder series
for $P(\nu',s')^\Gamma$, along with the fact that all $P(\nu',s')^\Gamma$
are projective, to obtain:
$$\begin{array}{l}
= \displaystyle t^{d_{\Psi}(\nu',\nu)}\sum_{j\geq 0}
(-1)^j\dim\Ext^j_{\curs{G}[\Gamma]}(P(\nu',s')^{\Gamma},V(\nu,s)) \\
= \displaystyle t^{d_{\Psi}(\nu',\nu)} \dim
\Hom_{\curs{G}[\Gamma]}(P(\nu',s')^{\Gamma},V(\nu,s)) \\
= t^{d_{\Psi}(\nu',\nu)}\delta_{(\nu',s'),(\nu,s)} =
\delta_{(\nu',s'),(\nu,s)}.
\end{array}$$

\noindent Thus, $H(E(B),-t)H(B,t)$ is the identity matrix, so $B =
\lie{B}(\Gamma)^{\op}$ is Koszul by \cite[Theorem 2.11.1]{BGS}.
\end{proof}

\subsection{}

We are now ready to approach the proof of \thmref{T1}. The following
Lemma will provide a major component of the proof.
Let $\pi_1 : \Lambda \to P^+$ be the projection map onto the first
coordinate. Recall that $A = \Sym V$.

\begin{lem}\label{L76}
Fix $(\mu,r) \in \Lambda$. Let $\Gamma \subset \Lambda$ be finite and
interval-closed with $(\mu,r) \popsi (\nu,s) \ \forall \ (\nu,s) \in
\Gamma$. Then, $\lie{B}(\Gamma)^{\op}$ has global dimension at most
$N_{\Psi}$, and
$$\lie{B}(\Gamma)^{\op} \cong \A_{\Psi}(\pi_1(\Gamma))^{\g}$$
as $\Z_+$-graded algebras.
\end{lem}

\begin{proof}
By definition,
$$1_{\nu}\A_{\Psi}^{\g}[d_{\Psi}(\mu,\nu)]1_{\mu} = (V(\nu)^* \otimes
\Sym^{d_{\Psi}(\mu,\nu)} V \otimes V(\mu))^{\g}.$$

\noindent For any finite-dimensional $\g$-modules, $V,W$, the map
$$\sum_i (f_i \otimes w_i) \to (v\mapsto \sum_if_i(v)w_i)$$

\noindent gives an isomorphism $(V^*\otimes W)^{\g} \cong
\Hom_{\g}(V,W)$. In particular,
$$(V(\nu)^* \otimes \Sym^{d_{\Psi}(\mu,\nu)} V \otimes V(\mu))^{\g} \cong
\Hom_{\g}(V(\nu),\Sym^{d_{\Psi}(\mu,\nu)} V \otimes V(\mu)).$$

\noindent Finally,
\begin{eqnarray*}
P(\mu,r)[r+d_{\Psi}(\mu,\nu)] & = & (\univ(\lie{a})\otimes_{\univ(\g)}
V(\lambda,r))[r+d_{\Psi}(\mu,\nu)]\\
& = & \Sym^{d_{\Psi}(\mu,\nu)} V \otimes V(\lambda,r)
\end{eqnarray*}

\noindent by \propref{P32}, and
$$\Hom_{\g}(V(\nu),\Sym^{d_{\Psi}(\mu,\nu)} V \otimes V(\mu)) \cong
\Hom_{\ghat}(P(\nu, r + d_{\Psi}(\mu,\nu)),P(\mu,r))$$

\noindent by \propref{P41}.

Notice that the product of the terms $1_{\nu}
\A_{\Psi}^{\g}[d_{\Psi}(\mu,\nu)] 1_{\mu}$ is from left to right, whereas
the composition of the Hom-spaces is from right to left. The bound on
global dimensions follows from \lemref{L75}.
\end{proof}

\subsection{} We are now able to prove our main result:

\begin{proof}[Proof of \thmref{T1}]
Notice first that $\lpsi \nu$ and $[\mu,\nu]_{\Psi}$ are finite and
interval-closed, so $\A_{\Psi}(\lpsi \nu)^{\g}$ and
$\A_{\Psi}([\mu,\nu]_{\Psi})^{\g}$ are Koszul and have finite global
dimension by \lemref{L76} and \thmref{T3}.

It remains to show the results for $\A_{\Psi}(\mu \lpsi)^{\g}$ and
$\A_{\Psi}^{\g}$. We begin by showing that the algebras in question have
finite global dimension. We only show the proof for $\A_{\Psi}^{\g}$; the
proof is similar for $\A_{\Psi}(\mu \lpsi)^{\g}$.

Suppose $\mu \in P^+$, and let $S_{\mu}$ be the simple left
$\A_{\Psi}^{\g}$-module corresponding to the idempotent $1_{\mu}$. Recall
that $1_{\nu}\A_{\Psi}^{\g} 1_{\mu} \neq 0$ only if $\mu \lpsi \nu$ by
\lemref{L76}. Thus, the projective cover of $S_{\mu}$ in the category of
finite-dimensional left $\A_{\Psi}^{\g}$-modules is
$$P_{\mu} := \A_{\Psi}^{\g} 1_{\mu} = \bigoplus_{\mu \lpsi \nu}
1_{\nu}\A_{\Psi}^{\g} 1_{\mu} = \A_{\Psi}(\mu \lpsi)^{\g} 1_{\mu}.$$

\noindent As in \propref{P41}, this yields that
\[ [P_{\nu}:S_{\mu}] > 0 \Longrightarrow \mu \lpsi \nu, \]

\noindent so we obtain a projective resolution of $S_{\mu}$ in the
category of finite-dimensional left $\A_{\Psi}(\mu \lpsi)^{\g}$-modules.
Applying $\Hom_{\A_{\Psi}^{\g}}(-,S_{\nu})$ to this projective resolution
and using \lemref{L76} and \lemref{L74}, the statements on the global
dimension follow from the result for $\A_{\Psi}(\lpsi \nu)^{\g}$ and
$\A_{\Psi}([\mu,\nu]_{\Psi})^{\g}$.

It remains to show that $\A_{\Psi}^{\g}$ and $\A_{\Psi}(\mu \lpsi)^{\g}$
are Koszul. The proof is finished by adapting the proof of the analogous
theorem in \cite{CG2} while keeping in mind that we use a different
definition for $\lpsi$ and the reverse ordering on the summands of ${\bf
A}$. A brief summary of the proof for $\A_{\Psi}^{\g}$ is provided for
the reader (see \cite{Ri} for more details):

Let ${\bf T}V = \vstar \otimes T(V) \otimes {\bf V}$. Use the kernel of
the canonical projection  $\Pi: {\bf T}V \to \A$ to show that
$\A_{\Psi}^{\g}$ is quadratic. Finally, show that the Koszul resolution
of $\A_{\Psi}^{\g}$ is exact, which shows that $\A_{\Psi}^{\g}$ is Koszul
by \cite[Theorem 2.6.1]{BGS}.
\end{proof}

We conclude by remarking that it is possible to construct a linear graded
resolution for the algebras addressed in \thmref{T1}. More generally,
such a resolution has been constructed for every Koszul algebra in
\cite[Theorem 2.6.1]{BGS}.

\end{document}